\documentclass{amsart}

\usepackage{amsfonts,amssymb,amscd,amsmath,enumerate,verbatim,calc,cite}

\newcommand{\rank}{\operatorname{rank}}
\newcommand{\id}{\operatorname{id}}

\newcommand{\reg}{\operatorname{reg}}

\newcommand{\ord}{\operatorname{ord}}

\newcommand{\Z}{\mathbb{Z}}

\theoremstyle{plain}
\newtheorem{theorem}{Theorem}

\newtheorem{lemma}[theorem]{Lemma}

\newtheorem{proposition}[theorem]{Proposition}

\theoremstyle{definition}

\theoremstyle{remark}

\newtheorem{Example}[theorem]{Example}

\newcounter{hours}\newcounter{minutes}
\newcommand{\printtime}{%
        \setcounter{hours}{\time/60}%
        \setcounter{minutes}{\time-\value{hours}*60}%
        \thehours\,h\ \theminutes\,min}

\begin{document}

\title[ Degree of reductivity of a modular  representation]{Degree of
  reductivity of a modular  representation}
\date{\today \printtime}

\author{ Martin Kohls}
\address{Technische Universit\"at M\"unchen \\
 Zentrum Mathematik-M11\\
Boltzmannstrasse 3\\
 85748 Garching, Germany}
\email{kohls@ma.tum.de}

\author{M\"uf\.it Sezer}
\address { Department of Mathematics, Bilkent University,
 Ankara 06800 Turkey}
\email{sezer@fen.bilkent.edu.tr}
\thanks{Second author is supported by a grant from T\"ubitak:112T113}


\begin{abstract}
For a finite dimensional representation $V$ of a group $G$ over a
field $F$, the degree of reductivity $\delta(G,V)$ is the smallest
degree $d$ such that every nonzero fixed point $v\in
V^{G}\setminus\{0\}$ can be separated from zero by a homogeneous
invariant of degree at most $d$.  We compute $\delta(G,V)$
explicitly for several classes of modular groups and
representations. We also demonstrate
that
 the maximal size of a cyclic subgroup is a sharp lower bound for
 this number in the case of modular abelian $p$-groups.
\end{abstract}

\maketitle

\section*{Introduction}

Separating points from zero by invariants is a classical problem in
invariant theory. While for infinite groups it is quite a problem to
describe those points where this is (not) possible (leading to the
definition of Hilbert's Nullcone), the finite group case is easier.
We fix the setup before going into details. Unless otherwise stated,
$F$ denotes an algebraically closed field of characteristic $p>0$.
We consider a finite dimensional representation $V$ of a finite
group $G$ over   $F$. We call $V$ a $G$-module. The action of $G$ on
$V$ induces an action of $G$ on $F[V]$ via
$\sigma(f):=f\circ\sigma^{-1}$ for $f\in F[V]$ and $\sigma\in G$.
Any homogeneous system of parameters (hsop) $f_{1},\ldots,f_{n}$ of
the invariant ring $F[V]^{G}$ has $\{0\}$ as its common zero set,
hence every nonzero point can be separated from zero by one of the
$f_{i}$. Moreover, Dade's algorithm \cite[section
3.3.1]{DerksenKemper} produces an hsop in degree $|G|$, hence every
nonzero point can be separated from zero by an invariant of degree
at most $|G|$. Therefore, for a given nozero point $v\in
V\setminus\{0\}$, the number
\[
\epsilon(G,v):=\min\{d>0\mid \text{ there is an }f\in F[V]^{G}_{d} \text{ such
that } f(v)\ne 0\}
\]
is bounded above by the group order $|G|$, and hence so is the
supremum $\gamma(G,V)$ of the  $\epsilon(G,v)$ taken over all $v\in
V\setminus\{0\}$.
There has been a recent interest in this number, see \cite{MR3071927,
  ElmerKohls1, ElmerKohls2}. In  \cite{ElmerKohls1}, another related
number $\delta(G,V)$ is introduced, which is defined to be zero if
$V^{G}=\{0\}$ and otherwise as the supremum of all $\epsilon(G,v)$
taken over all nonzero \emph{fixed} points $v\in
V^{G}\setminus\{0\}$. We propose the name \emph{degree of
reductivity} for $\delta(G,V)$. Note that a group $G$ is called
reductive, if for every $V$ and every $v\in V^{G}\setminus\{0\}$,
there exists a homogeneous positive degree invariant $f\in
F[V]^{G}_{+}$ such that $f(v)\ne 0$, hence the suggested name. It
was shown in \cite{ElmerKohls1}, that $\delta(G)$, the supremum of
the $\delta(G,V)$ taken over all $V$, equals the size of a Sylow-$p$
subgroup of $G$. The goal of this paper is to give more precise
information on $\delta(G,V)$ and compute it explicitly for several
classes of modular groups (i.e., $|G|$ is divisible by $p$) and
representations. In Section 1 we show that for a cyclic $p$-group
$G$ and every faithful $G$-module $V$, we have $\delta(G,V)=|G|$. In
that situation we compute $\epsilon(G,v)$ for every $v\in
V^{G}\setminus\{0\}$ as well. The most important stepstone that we
lay to our main results is a restriction of the degrees of certain
monomials that appear in invariant polynomials. We think that this
restriction can also be useful for further studies targeting the
generation of the invariant ring.
In Section 2 we consider an abelian  $p$-group $G$ and show that the
maximal size of a cyclic subgroup of $G$ is a lower bound for
 $\delta(G,V)$ for every faithful $G$-module $V$. We also
work out the  Klein four group and compute the $\delta$- and
$\gamma$-values for all its representations. It turns out that our
lower bound is sharp for a large number of these representations. In
the final section we deal with groups  whose order is divisible by
$p$ only once and put a squeeze on the $\delta$-values of the representations of these groups.

For a general reference for invariant theory we refer the reader to
\cite{MR1249931, MR2759466, DerksenKemper, MR1328644}.

\section{Modular cyclic groups}

  Let $G=Z_{p^{r}}$ be the cyclic group of
order $p^r$. Fix a generator $\sigma$ of $G$. It is well known that
there are exactly $p^r$ indecomposable $G$-modules $V_1, \dots ,
V_{p^r}$ over  $F$, and each indecomposable module $V_i$ is afforded
by $\sigma^{-1}$ acting via a Jordan block of dimension $i$ with
ones on the diagonal. Let $V$ be an arbitrary $G$-module over $F$.
Write
\[
\quad\quad\quad V=\bigoplus_{j=1}^k V_{n_j} \quad\quad \text{ (with
} 1\le n_{j}\le p^{r} \text { for all } j),
\]
where each $V_{n_j}$ is spanned as a vector space by $e_{1,j}, \dots
, e_{n_j, j}$. Then the action of $\sigma^{-1}$ is given by
$\sigma^{-1}(e_{i,j})=e_{i,j}+e_{i+1,j}$ for $1\le i < n$ and $\sigma^{-1}
(e_{n_j,j})=e_{n_j,j}$. Note that the fixed point space $V^G$ is
$F$-linearly spanned by $e_{n_1,1}, \dots ,e_{n_k,k}$. The dual $V_{n_j}^*$ is
isomorphic to $V_{n_j}$. Let  $x_{1,j},
\dots x_{n_j,j}$ denote the corresponding dual basis, then we have
\[
F[V]=F[x_{i,j} \; \mid \,\, 1\le i\le n_j, \,\,\, 1\le j\le k],
\]
and the action of $\sigma$ is given by
$\sigma(x_{i,j})=x_{i,j}+x_{i-1,j}$ for $1<i\le n_j$ and
$\sigma(x_{1,j}) =x_{1,j}$ for $1\le j\le k$. We call the $x_{n_{j},j}$ for
$1\le j\le k$ \emph{terminal variables}.
Set $\Delta=\sigma-1$.  Notice that $\Delta (x_{i,j})=x_{i-1,j}$ if
$i\ge 2$ and $\Delta (x_{1,j})=0$. Since $\Delta (f)=0$
for $f\in F[V]^G$, and $\Delta$ is an additive map, we have the
following, see also the discussion in
\cite[before Lemma 1.4]{SezerShank}.
\begin{lemma} \label{first}
  Let $f\in F[V]^G$ and $M$ be a monomial that
  appears in $f$. If a monomial $M'$ appears in $\Delta (M)$, then there is another monomial $M''\neq
  M$ that appears in $f$ such that $M'$  appears in $\Delta
  (M'')$ as well.
\end{lemma}

We say that a monomial $M$ lies above $M'$ if $M'$ appears in
  $\Delta(M)$.
We will use the well-known Lucas-Theorem on binomial coefficients modulo a
prime in our computations (see \cite{MR0023257} for a short proof):

\begin{lemma}[Lucas-Theorem]
Let $s, t$ be integers with base-$p$-expansions
$t=c_mp^m+c_{m-1}p^{m-1}+\cdots +c_0$ and
$s=d_mp^m+d_{m-1}p^{m-1}+\cdots +d_0$, where $0\le c_i,d_i\le p-1$
for  $0\le i\le m$. Then ${t \choose s}\equiv \prod _{0\le i\le
m}{c_i\choose d_i}  \mod p$.
\end{lemma}

 The following lemma is the main technical stepstone for the rest of the paper.

\begin{lemma}\label{divide}
For $0\le s\le r$, define
\[
J_{s}=\{j\in\{1,\ldots,k\}\mid n_{j}>p^{s-1}\}.
\]
Let $M=\prod_{1\le j\le k}x_{n_j,j}^{a_j}$ be
a monomial consisting only of terminal variables, that appears in an invariant polynomial with nonzero
coefficient. Then $p^s$ divides $a_j$ for all $j\in J_s$.
\end{lemma}

\begin{proof}
As the case $s=0$ is trivial, we will assume $s\ge 1$ from now.
Let $f\in F[V]^{G}$ be an invariant polynomial  in which $M$ appears
with a nonzero coefficient, and $j\in J_{s}$. Without loss of
generality, we assume $j=1$ and $a_{1}\ne 0$. Set $M'=\prod_{2\le j\le
k}x_{n_j,j}^{a_j}$. For simplicity we denote $a_1$ with $a$. Then
$M=x_{n_1,1}^aM'$, and the claim is $p^{s}|a$. We proceed by
induction on $s$ and at each step we verify the claim for all $r$ such that $s\le r$.
    Assume $s=1$ and  $r\ge 1=s$.  By way of contradiction, we assume
$p\not|a$. Then we can write $a=c_1p+c_0$, where $c_1$ and $c_0$ are
non-negative integers with $1\le c_0<p$. We have $\sigma
(M)=\sigma(x_{n_1,1}^a)\sigma (M')$ and
$\sigma(x_{n_1,1}^a)=(x_{n_1,1}+x_{n_1-1,1})^a$. Since $M'$ appears
in $\sigma(M')$ with coefficient one, it follows that the coefficient
of $x_{n_1-1,1}x_{n_1,1}^{a-1}M'$ in $\sigma(M)$ is ${a\choose
1}=a\equiv c_{0}\not\equiv 0\mod p$.  Therefore
$x_{n_1-1,1}x_{n_1,1}^{a-1}M'$ appears in $\sigma(M)-M=\Delta (M)$.
As $M=x_{n_1,1}^aM'$ only consists of terminal variables, it can be
seen easily that it is the only monomial lying above
$x_{n_1-1,1}x_{n_1,1}^{a-1}M'$, which is a contradiction by Lemma
\ref{first}.

Next assume that $s>1$ and let $r\ge s$ be arbitrary. Note that the
induction hypothesis is
 that the assertion holds for every pair $r',s'$ with    $1\le s'\le r'$ and $s'< s$.  Consider the base-$p$-expansion
$a=c_lp^l+c_{l-1}p^{l-1}+\cdots +c_0p^{0}$ of $a$ where $0\le c_l,
\dots , c_0 \le p-1$. Let $t$ denote the smallest integer such that
$c_t\neq 0$. We claim that $p^{s}|a$, which is equivalent to $t\ge
s$. By way of contradiction  assume $t<s$. Define $b=a-p^t$. Then
the base-$p$-expansion of $b$ is $c_lp^l+\cdots c_{t+1}p^{t+1}
+(c_t-1)p^t+0\cdot p^{t-1}+ \dots +0\cdot p^{0}$.  As in the basis
case, we see that the coefficient of
$x_{n_1-1,1}^{p^t}x_{n_1,1}^{a-p^t}M'$ in $\sigma
(M)=(x_{n_1,1}+x_{n_1-1,1})^a\sigma (M')$ is ${a\choose p^{t}}$. By
the Lucas-Theorem, ${a \choose p^{t}}\equiv {c_t\choose
1}=c_{t}\not\equiv 0 \mod p$. So
$x_{n_1-1,1}^{p^t}x_{n_1,1}^{a-p^t}M'$ appears in $\Delta(M)$. By
Lemma \ref{first} there exists another monomial $M''$ in $f$ that
lies above $x_{n_1-1,1}^{p^t}x_{n_1,1}^{a-p^t}M'$. We have
$M''=x_{n_1-1,1}^{d}x_{n_1,1}^{a-d}M'$ for some $1\le d<p^t$. Since
$a-p^t<a-d<a$ and $p^t$ divides $a$ it follows that
\[
p^t \text{  does not divide } a-d \quad (*).
\]
Let $H$ denote the subgroup of $G$ generated by $\sigma^p$. Note
that $H\cong Z_{p^{r-1}}$ and consider  $V_{n_{1}}$ as an
$H$-module. From $\sigma^{p}-1=(\sigma-1)^{p}$ it follows  that
$V_{n_{1}}$ decomposes into $p$ indecomposable $H$-modules such that
$x_{n_1,1}, x_{n_1-1,1}, \dots , x_{n_1-p+1,1}$ become terminal
variables with respect to the $H$-action. Note that by assumption,
$r\ge s\ge 2$ and as $1=j\in J_{s}$, we have $n_{1}>p^{s-1}\ge p$.
Also the $H(\cong Z_{p^{r-1}})$-module generated by $x_{n_1,1}$ has
dimension $\lceil\frac{n_{1}}{p}\rceil>p^{s-2}$. Therefore the
monomial $M''=x_{n_1,1}^{a-d}\cdot x_{n_1-1,1}^{d}M'$
appearing in $f\in F[V]^{G}\subseteq F[V]^{H}$ consists only of
terminal variables with respect to the $H(\cong Z_{p^{r-1}})$-action
and $x_{n_{1},1}$ is a terminal variable whose index would appear in
the set $J_{s-1}'$ corresponding to the considered $H(\cong
Z_{p^{r-1}})$-action. Therefore, the induction hypothesis (with
$s'=s-1$ and $r'=r-1$) applied to $M''$ yields $p^{s-1}|a-d$. As we
have assumed $t<s$, it follows that $p^{t}$ divides $a-d$, which is
a contradiction to (*) above.
\end{proof}

With this lemma we can precisely compute the degree required to
separate a nonzero fixed point from zero.

\begin{theorem}
Let $v=\sum_{1\le j\le k}c_je_{n_j,j}\in V^{G}\setminus\{0\}$ be a nonzero fixed point,
where $c_{1},\ldots,c_{k}\in F$. Let $J$
denote the set of all $j\in\{1,\ldots,k\}$ such that $c_j\neq 0$, and $s$ denote
the maximal integer  such that $p^{s-1}<n_j$ for all $j\in J$. Then
$\epsilon (G,v)=p^s$.

In particular, if $V$ is a faithful $G$-module, then $\delta
(G,V)=p^r$.
\end{theorem}

\begin{proof}
Any homogeneous invariant polynomial of positive degree that is
nonzero on $v$ must contain a monomial $M$ with a nonzero
coefficient in the variables of the set $\{x_{n_j,j}\mid j \in J\}$.
With $s$ as defined above, by the previous lemma the exponents of
the $x_{n_{j},j}$ in $M$ are divisible by $p^s$ for all $j\in
J\subseteq J_{s}$. Hence $\epsilon (G,v)\ge p^s$, so it remains to
prove the reverse inequality. The maximality condition on $s$
implies the existence of a $j'\in J$ such that $p^{s-1}<n_{j'}\le
p^s$. Then the Jordan block representing the action of
$\sigma$ on $x_{1,j'}, \dots x_{n_{j'},j'}$ has order $p^s$, and so
the orbit product $N=\prod_{m \in Gx_{n_{j'},j'}}m\in F[V]_{+}^{G}$
is an invariant homogeneous polynomial of degree $p^s$. Furthermore,
for every $\sigma\in G$ and the corresponding element
$m=\sigma(x_{n_{j'},j'})\in Gx_{n_{j'},j'}$ in the orbit, we have
$m(v)=(\sigma(x_{n_{j'},j'}))(v)=x_{n_{j'},j'}(\sigma^{-1}v)=x_{n_{j'},j'}(v)=c_{j'}$,
where we used $v\in V^{G}$. Hence, $N(v)=c_{j'}^{p^{s}}\ne 0$, which
shows $\epsilon (G,v)\le p^s$.

For the final statement, note that if $V$ is a faithful $G$-module,
then there is a $j'\in\{1,\ldots,k\}$ satisfying $p^{r-1}< n_{j'}\le
p^r$. Now for $v=e_{n_{j'},j'}\in V^{G}\setminus\{0\}$, in the
notation above we have $J=\{j'\}$ and $s=r$, so the first part
yields $\epsilon (G,v)=p^r=|G|$. It follows $\delta(G,V)=|G|$ as
claimed.
\end{proof}

 We now consider the general modular cyclic group $\tilde{G}=Z_{p^{r}m}$, where $m$ is a non-negative integer with
$(p,m)=1$. Let $G$ and $N$ be the subgroups of $\tilde{G}$ of order
$p^{r}$ and $m$, respectively. Fix a generator $\sigma$  of $G$ and
a generator $\alpha$  of $N$. For every $1\le n\le p^r$ and an
$m$-th root of unity $\lambda\in F$, there is an $n$-dimensional
$\tilde{G}$-module $W_{n,\lambda}$ with basis $e_1, e_2, \dots ,
e_n$ such that $\sigma^{-1} (e_i)=e_i+e_{i+1}$ for $1\le i\le n-1$,
$\sigma^{-1} (e_n)=e_n$ and $\alpha (e_i)=\lambda e_i$ for $1\le
i\le n$. It is well-known that the $W_{n,\lambda}$ form the complete
list of indecomposable $\tilde{G}$-modules, see \cite[Lemma
3.1]{KohlsSezerKleinFour} for a proof.

Notice that the indecomposable module $W_{n,\lambda}$ is faithful if and only
if $p^{r-1}<n\le p^r$ and $\lambda$ is a primitive $m$-th
root of unity. Let $x_{1}, \dots , x_{n}$ denote the corresponding basis
for $W_{n,\lambda}^*$. We have an isomorphism $W_{n,\lambda}^*\cong
W_{n,\lambda^{-1}}$, where the action of $\sigma$ on  $x_1, \dots , x_n$ is
given by an upper diagonal
Jordan block. Note that if $\lambda\ne 1$, we have $W_{n,\lambda}^{\tilde{G}}=\{0\}$, and so $\delta
(\tilde{G},W_{n,\lambda})=0$.

\begin{proposition}
Let $\tilde{G}=Z_{p^{r}m}$ and  $W_{n,\lambda}$ be a faithful
indecomposable $\tilde{G}$-module, i.e. $p^{r-1}<n\le p^r$ and $\lambda\in F$ is a primitive $m$-th
root of unity. Then $\gamma (\tilde{G},W_{n,\lambda})=|\tilde{G}|=p^{r}m$.
\end{proposition}

\begin{proof}
Let $f\in F[W_{n,\lambda}]_{+}^{\tilde{G}}$ be a homogeneous
invariant of positive degree $d$ such that $f(e_n)\neq 0$. Then $f$
contains the monomial $x_n^d$ with a nonzero coefficient.
Considered as a $G$-module, $W_{n,\lambda}$ is isomorphic to the
indecomposable $G$-module $V_{n}$. Since $f$ is particularly
$G$-invariant, we get from Lemma \ref{divide} that $p^r$ divides
$d$. As $f$ is also $\alpha$-invariant, and $\alpha$ acts just by
multiplication with $\lambda^{-1}$ on every variable, it follows
that $x_n^d$ is $\alpha$-invariant, hence we have $\lambda^{d}=1$.
As $\lambda$ is a primitive $m$-th root of unity, it follows that
$m$ divides $d$. Since $p^r$ and $m$ are coprime we get that
$p^{r}m$ divides $d$. Therefore $\gamma (\tilde{G},W_{n,\lambda})\ge
\epsilon(\tilde{G},e_{n})\ge p^{r}m=|\tilde{G}|$. The reverse
inequality always holds by Dade's hsop algorithm.
\end{proof}

Now let $1\le n \le p^{r}$ be arbitrary and $\lambda\in F$ be an
arbitrary $m$th root of unity. Define $0\le s\le r$ such that
$p^{s-1}<n\le p^{s}$ and let $m'$ denote the order of $\lambda$ as
an element of the multiplicative group $F^{\times}$ (then $m'|m$).
Then $W_{n,\lambda}$ can be considered as a faithful
$Z_{p^{s}m'}$-module, hence the result above yields
$\gamma(\tilde{G},W_{n,\lambda})=|Z_{p^{s}m'}|=p^{s}m'$. As the
$\gamma$-value of a direct sum of modules is the maximum of the
$\gamma$-values of the summands (see fore example \cite[Proposition
3.3]{ElmerKohls1}), the proposition above allows to compute
$\gamma(\tilde{G},V)$ for every $\tilde{G}$-module. This precisises
the result of \cite[Corollary 4.2]{ElmerKohls1}, which states
$\gamma(\tilde{G})=|\tilde{G}|$. As an interesting example, take
again $\lambda$ a primitive $m$th root of unity and consider the
$\tilde{G}$-module $V:=W_{p^{r},1}\oplus W_{1,\lambda}$. Note that
though $V$ is a faithful $\tilde{G}$-module, we get from the above
\[
\gamma(\tilde{G},V)=\max\{\gamma(\tilde{G},W_{p^{r},1}),\gamma(\tilde{G},W_{1,\lambda})\}=\max\{p^{r},m\}
\]
which is strictly smaller than $|\tilde{G}|=p^{r}m$ if $r>0$ and $m>1$.

\section{Modular abelian $p$-groups}

Before we focus on abelian $p$-groups, we start with a more general lemma.

\begin{lemma}\label{lemmapgroup}
Let $\tilde{G}$ be a $p$-group, $V$ a faithful $\tilde{G}$-module and let $\sigma\in \tilde{G}$ be of order $p^{r}$ such
that $\sigma^{p^{r-1}}\in Z(\tilde{G})$ (the center of $\tilde{G}$). Then  $\delta
(\tilde{G},V)\ge p^r$.
\end{lemma}

\begin{proof}
Let $G$ denote the subgroup of $\tilde{G}$ generated by $\sigma$. We
follow the notation of the previous section and consider the
decomposition $V=\bigoplus_{j=1}^k V_{n_j}$ of $V$ as a $G$-module.
Since $V$ is also  faithful as $G$-module, we have $J:=J_{r}=
\{j\in\{1,\ldots,k\}\mid n_j>p^{r-1} \}\neq \emptyset$. We can
choose a suitable basis of $V$ such that $\sigma^{-1}$
 acts on this basis via sums of Jordan blocks of dimensions
 $n_{1},\ldots,n_{k}$. Set $\Gamma =\sigma^{-1}-1$.  Let $W$ denote the image of the map
$\Gamma^{p^{r-1}}$ on $V$. Since $\sigma^{p^{r-1}}\in Z(\tilde{G})$,
we have that $\Gamma^{p^{r-1}}$ commutes with every
$\tau\in\tilde{G}$, hence $W$ is a $\tilde{G}$-module.  We also have
$W\subseteq \oplus_{j\in J} V_{n_j}$ because $\Gamma^{p^{r-1}}
V_{n_j}=\{0\}$ if $n_j\le p^{r-1}$. On the other hand,
$\Gamma^{p^{r-1}} V_{n_j}$ is spanned by $e_{p^{r-1}+1,j},\,
e_{p^{r-1}+2,j},\,\ldots, e_{n_j,j}$ for $n_j>p^{r-1}$. But $J\neq \emptyset$, so
we get that $W \neq \{0\}$ and in  particular   $e_{n_j,j}\in W$ for
$j\in J$. Hence $W^G$ is spanned $F$-linearly by $\{e_{n_j,j} \mid
j\in J\}$.  Moreover, since every modular action of a $p$-group on a
nonzero module has a non-trivial fixed point, we have
\[\{ 0\} \neq W^{\tilde{G}}\subseteq W^G=\langle \{e_{n_j,j} \mid
j\in J\}\rangle.\] Choose any nonzero vector $v\in W^{\tilde{G}}
\subseteq V^{\tilde{G}}$. As $v$ is in the span of $\{e_{n_j,j} \mid
j\in J\}$, every homogeneous polynomial $f\in
F[V]^{\tilde{G}}\subseteq F[V]^{G}$ of positive degree that is nonzero on $v$ must contain a monomial with nonzero coefficient in the
variables $\{x_{n_{j},j}\mid j\in J\}$. Since $f$ is also
$G$-invariant, Lemma \ref{divide} applies and we get that the
exponents of these variables in this monomial  are all divisible by
$p^r$. It follows $\delta(\tilde{G},V)\ge\epsilon(\tilde{G},v) \ge
p^{r}$ as desired.
\end{proof}

In the following two examples, $F$ is an algebraically closed field of
characteristic~$2$.

\begin{Example}
Consider the dihedral group $D_{2^{r+1}}=\langle \sigma,\rho\rangle$ of order
$2^{r+1}$ with relations $\ord(\sigma)=2$, $\ord(\rho)=2^{r}$ and
$\sigma\rho\sigma^{-1}=\rho^{-1}$. Then $\rho^{2^{r-1}}\in
Z(D_{2^{r+1}})$. Hence the lemma applies, and for every faithful
$D_{2^{r+1}}$-module $V$ we have $\delta(D_{2^{r+1}},V)\ge 2^{r}$.
\end{Example}

\begin{Example}
Consider the quaternion group $Q$ of order $8$. There is an element $\sigma\in
Q$ of order $4$ such that $\sigma^{2}\in Z(Q)$. From the lemma it follows that
for every faithful $Q$-module $V$, we have $\delta(Q,V)\ge 4$.
\end{Example}

\begin{Example}
  Consider
the non-abelian group
\[
T_{p}:=\left\{\left(\begin{array}{cc}
\overline{a}&\overline{b}\\
\overline{0}&\overline{1}
\end{array}\right)\in(\Z/p^{2}\Z)^{2\times 2}\mid a,b\in\Z,\,\,\, a\equiv 1 \mod p\right\}
\]
of order $p^{3}$ (where we write $\overline{a}:=a+p^{2}\Z$). Note
that $T_{2}\cong D_{8}$. The element
$\sigma:=\left(\begin{array}{cc}
\overline{1}&\overline{1}\\
\overline{0}&\overline{1}
\end{array}\right)\in T_{p}$ is of order $p^{2}$, and it can be checked easily
that $\sigma^{p}\in Z(T_{p})$. From the lemma it follows that for every faithful
$T_{p}$-module $V$, we have $\delta(T_{p},V)\ge p^{2}$.
\end{Example}

Recall that that for a group $\tilde{G}$, the exponent
$\exp(\tilde{G})$ of $\tilde{G}$ is the least common multiple of the
orders of its elements. In particular for an abelian group, the
exponent is the maximal order of an element. As a corollary of the
above lemma, we get:

\begin{theorem}\label{abelianpgroup}
Let $\tilde{G}$ be a non-trivial $p$-group. Then for every faithful
$\tilde{G}$-module $V$ we have  \[\delta (\tilde{G},V)\ge \exp(Z(\tilde{G}))\ge p.\]
If  $\tilde{G}$ is an abelian $p$-group, we particularly have \[\delta (\tilde{G},V)\ge \exp(\tilde{G})\ge p.\]
\end{theorem}

\begin{proof}
First note that for $p$-groups, its center is non-trivial, so particularly we
have
$\exp(Z(\tilde{G}))\ge p$. Now chose an element $\sigma\in Z(\tilde{G})$ of
maximal order $p^{r}=\exp(Z(\tilde{G}))$. Then Lemma \ref{lemmapgroup} applies
and yields $\delta(\tilde{G},V)\ge p^{r}=\exp(Z(\tilde{G}))$. Finally, if $\tilde{G}$
is an abelian $p$-group, we have $\tilde{G}=Z(\tilde{G})$.
\end{proof}

For a recent related study of the invariants of abelian $p$-groups
we refer the reader to \cite{MR3022756}. We also remark that the
inequality in  Theorem \ref{abelianpgroup} is sharp, see
 Theorem \ref{deltasKleinFour} below.

\subsection*{The Klein four group}

Let $\tilde{G}$ denote the Klein four group with generators $\sigma_{1}$ and
$\sigma_{2}$, and $F$ an algebraically closed field of characteristic
$2$. The goal of this section is to compute the $\delta$- and $\gamma$-value of every
$\tilde{G}$-module (in all cases, both numbers are equal here). We first  give the $\delta/\gamma$-value for each
indecomposable representation of the Klein four group.
 The complete list of indecomposable representations is for example
given in \cite[Theorem 4.3.3]{MR1644252}. There, the indecomposable
representations are classified in five types (i)-(v), and we will
use the same enumeration. For the notation of the modules, we follow
\cite{KohlsSezerKleinFour} but note that there types (iv) and (v)
are interchanged.  The first {\bf type (i)} is just the regular
representation $V_{\reg}:=F\tilde{G}$, and here we have
$\delta(\tilde{G},V_{\reg})=\gamma(\tilde{G},V_{\reg})=4=|\tilde{G}|$
by \cite[Theorem 1.1 and Proposition 2.4]{ElmerKohls1}.

The {\bf type (ii)} representations $V_{2m,\lambda}$ are
parameterized by a positive integer $m$ and $\lambda \in F$. Then
$V_{2m,\lambda}$ is defined as the $2m$-dimensional representation
spanned by $e_1,\ldots,e_m,$$h_1,\ldots,h_m$ such that the action is
given by $\sigma_i(e_j)=e_j$, $\sigma_1(h_j)=h_j+e_j$ for $i=1,2$
and $j=1,\ldots,m$, $\sigma_2(h_j)=h_j+\lambda e_j+e_{j+1}$ for
$1\le j<m$ and $\sigma_2(h_m)=h_m+\lambda e_m$. Let $x_1, \dots
,x_m, y_1, \dots y_m$ be the
  elements of $V_{2m,\lambda}^*$ corresponding to $h_1, \dots ,h_m, e_1, \dots e_m$.
Then we have
 $\sigma_i(x_j)=x_j$,
$\sigma_1(y_j)=y_j+x_j$ for $i=1,2$ and
$j=1,\ldots,m$, $\sigma_2(y_1)=y_1+\lambda x_1$ and
$\sigma_2(y_j)=y_j+\lambda x_j+x_{j-1}$ for $1<j\le m$.

\begin{lemma}\label{lemmatypeii}
In the notation as above, we have that $\delta(\tilde{G},V_{2,\lambda})=\gamma(\tilde{G},V_{2,\lambda})$ equals $2$
if $\lambda\in\{0,1\}$, and it equals $4$ if $\lambda\in F\setminus\{0,1\}$.
\end{lemma}

\begin{proof}
If $\lambda\in\{0,1\}$, the corresponding matrix group is of order $2$, and
the result follows easily. If $\lambda\in F\setminus\{0,1\}$ it follows from \cite[Theorem 3.7.5]{DerksenKemper} that $F[V_{2,\lambda}]^{\tilde{G}}$ is  generated
by $x_{1}$ and the norm $N_{\tilde{G}}(y_{1})$, as those two invariants form
an hsop and the product of their degrees equals the group order $4$. Now the
claim follows easily.
\end{proof}

\begin{proposition}\label{typeii}
In the notation as above, we have $\delta(\tilde{G},V_{2m,\lambda})=\gamma(\tilde{G},V_{2m,\lambda})=4$ for all
$m\ge 2$ and $\lambda\in F$.
\end{proposition}

\begin{proof}
We have $\delta(\tilde{G},V_{2m,\lambda})\le\gamma(\tilde{G},V_{2m,\lambda})\le 4$, from Dade's hsop algorithm,
hence it is enough to show $\delta(\tilde{G},V_{2m,\lambda})\ge 4$.
Consider the point $e_{m}\in V_{2m,\lambda}^{\tilde{G}}\setminus\{0\}$. Any homogeneous invariant $f\in
F[V_{2m,\lambda}]^{\tilde{G}}_{d}$ of positive degree $d$ separating $e_{m}$ from
zero must contain $y_{m}^{d}$. Lemma \ref{Klein} implies $d\ge 4$, so
$\delta(\tilde{G},V_{2m,\lambda})\ge \epsilon(\tilde{G},e_{m})\ge 4$,
finishing the proof.
\end{proof}

Set $\Delta_i=\sigma_i-1$ for $i=1,2$. Since $\Delta_i (f)=0$ for
every polynomial $f\in F[V]^{\tilde{G}}$, the assertion of Lemma 1
holds for $\Delta=\Delta_i$ for $i=1,2$. We say that a monomial $M$
lies above the monomial $M'$ with respect to $\Delta_i$ if $M'$
appears in $\Delta_i(M)$.

\begin{lemma}\label{Klein}
Assume that $V=V_{2m,\lambda}$ with $m\ge 2$. Then
$y_m^d$ does not appear in a polynomial in $F[V]^{\tilde{G}}$ for
$1\le d\le 3$.
\end{lemma}

\begin{proof}
Assume  that $y_m^d$ appears in $f\in F[V]^{\tilde{G}}$. Since
$\{x_m, y_m\}$ spans a two dimensional indecomposable summand as a
$\langle \sigma_1\rangle$-module, Lemma \ref{divide} applies and we
get that $d$ is divisible by $2$. Assume on the contrary that $d=2$.
Then $\sigma_1(y_m^2)=y_m^2+x_m^2$. So $x_m^2$ appears in
$\Delta_1(y_m^2)$. Since $y_{m}x_{m}$ is the only other monomial in
$F[V]$ that lies above $x_m^2$ with respect to $\Delta_1$ we get
that $y_{m}x_{m}$ appears in $f$ as well. Moreover since the
coefficient of $x_m^2$ in $\Delta_1(y_m^2)$ and
$\Delta_1(y_{m}x_{m})$ is one, it follows that the coefficients of
$y_m^2$ and $y_{m}x_{m}$ in $f$ are equal. Call this nonzero
coefficient $c$. Then the coefficient of $x_m^2$ in
$\Delta_2(cy_m^2+cy_{m}x_{m})$ is $c(\lambda^2+\lambda)$.
 Since
$y_m^2$ and $y_mx_m$ are the only monomials in $F[V]$ that lie
above $x_m^2$ with respect to $\Delta_2$, we get that
$\Delta_2(f)\neq 0$ if $\lambda\neq 0,1$, giving a contradiction.
Next assume that $\lambda=0$. Then, since $y_mx_m$ appears in $f$
and $\sigma_2(y_mx_m)=(y_m+x_{m-1})x_m$ we get that $x_{m-1}x_m$
appears in $\Delta_2 (y_mx_m)$. This gives a contradiction by
Lemma \ref{first} again because $y_mx_m$ is the only monomial that
lies above $x_{m-1}x_{m}$ with respect to $\Delta_2$. Finally, we
note that the cases $\lambda=1$ and $\lambda=0$ correspond to the
same matrix group and so their invariants are the same.
\end{proof}

The {\bf type (iii)} representations $W_{2m}$ are $2m$-dimensional
representations ($m\ge 1$) which are obtained from $V_{2n,0}$ just
by interchanging the actions of $\sigma_{1}$ and $\sigma_{2}$. In
particular, $W_{2m}$ and $V_{2n,0}$ have the same invariant ring, so
we get as a corollary from Lemma \ref{lemmatypeii} and Proposition
\ref{typeii} that
$\delta(\tilde{G},W_{2})=\gamma(\tilde{G},W_{2})=2$ and
$\delta(\tilde{G},W_{2m})=\gamma(\tilde{G},W_{2m})=4$ for all $m\ge
2$.

The {\bf type (iv)} representations $V_{2m+1}$ for $m\ge 1$ are
$2m+1$-dimensional representations. (Note that in
\cite{KohlsSezerKleinFour}, these representations are listed as type
(v).) They are linearly spanned by
$e_{1},\ldots,e_{m},h_{1},\ldots,h_{m+1}$, where
$\sigma_{i}(e_{j})=e_{j}$ for $i=1,2$ and $1\le j\le m$,
$\sigma_{1}(h_{i})=h_{i}+e_{i}$ for $1\le i\le m$,
$\sigma_{1}(h_{m+1})=h_{m+1}$, $\sigma_{2}(h_{1})=h_{1}$, and
$\sigma_{2}(h_{i})=h_{i}+e_{i-1}$ for $2\le i\le m+1$. Let
$x_{1},\ldots,x_{m+1},y_{1},\ldots,y_{m}$ be the elements of
$V_{2m+1}^{*}$ corresponding to
$h_{1},\ldots,h_{m+1},e_{1},\ldots,e_{m}$. Then we have
$\sigma_{i}(x_{j})=x_{j}$ for $i=1,2$ and $1\le j\le m+1$,
$\sigma_{1}(y_{j})=y_{j}+x_{j}$ and
$\sigma_{2}(y_{j})=y_{j}+x_{j+1}$ for $1\le j\le m$.

\begin{proposition}
We have $\delta(\tilde{G},V_{2m+1})=\gamma(\tilde{G},V_{2m+1})=4$ for all $m\ge 1$.
\end{proposition}

\begin{proof}
Again by Dade's hsop-algorithm, we have
$\delta(\tilde{G},V_{2m+1})\le\gamma(\tilde{G},V_{2m+1})\le 4$.
Consider the point $e_{m}\in V_{2m+1}^{\tilde{G}}$, and let $f\in
F[V_{2m+1}]^{\tilde{G}}$ be homogeneous of minimal positive degree
$d$ such that $f(e_{m})\ne 0$. Then $y_{m}^{d}$ must appear in $f$
with a nonzero coefficient. Since $\{x_m, y_m\}$ spans a two
dimensional indecomposable summand as a $\langle
\sigma_1\rangle$-module and $f$ is also $\langle
\sigma_1\rangle$-invariant,  Lemma \ref{divide} applies and we get
that $d$ is divisible by $2$. By \cite[Proposition
5.8.3]{SezerShank}, $y_{m}^{2}$ does not appear in a
$\tilde{G}$-invariant polynomial. It follows $d\ge 4$, so we are
done.
\end{proof}

The {\bf type (v)} representations $W_{2m+1}$ for $m\ge 1$ are $2m+1$-dimensional
representations. (Note that in \cite{KohlsSezerKleinFour}, these
representations are given type (iv).)
They are afforded by $\sigma_{1}\mapsto
\left(\begin{array}{c|c} I_{m+1}&  \begin{array}{c} I_{m} \\\hline
0_{1\times m}\end{array}  \\ \hline 0& I_{m}
  \end{array}\right)
$ and $\sigma_{2}\mapsto \left(\begin{array}{c|c} I_{m+1}&  \begin{array}{c} 0_{1\times m} \\\hline
I_{m}\end{array}  \\ \hline 0& I_{m}
  \end{array}\right)$, where $0_{k\times l}$ denotes a $k\times l$ matrix
whose entries are all zero. In \cite[section 4]{SezerShank} (with notation $F[W_{2m+1}]=:F[y_{1},\ldots,y_{m+1},x_{1},\ldots,x_{m}]$), an hsop
consisting of invariants of degree at most $2$ is given for
$F[W_{2m+1}]^{\tilde{G}}$. As the $\delta$-value is clearly not one, it
follows $\delta(\tilde{G},W_{2m+1})=\gamma(\tilde{G},W_{2m+1})=2$ for all $m\ge 1$.

\begin{theorem}\label{deltasKleinFour}
Let $V$ be a non-trivial representation of the Klein four group
$\tilde{G}$ over an algebraically closed field of characteristic $2$, and consider its
decomposition into indecomposable summands. Then
$\delta(\tilde{G},V)=\gamma(\tilde{G},V)=2$ if and only if every
non-trivial indecomposable summand is isomorphic to one of
$V_{2,0}$, $V_{2,1}$, $W_{2}$ or $W_{2m+1}$ $(m\ge 1)$. If another
non-trivial indecomposable summand appears, then
$\delta(\tilde{G},V)=\gamma(\tilde{G},V)=4$.
\end{theorem}

\begin{proof}
The $\delta/\gamma$-value of a direct sum equals the maximal
$\delta/\gamma$-value of a summand (see \cite[Proposition
2.2/Proposition 3.3]{ElmerKohls2}), so the theorem follows from the
values for the indecomposable modules above.
\end{proof}

\section{Groups of order with simple prime factor $p$}

In this section we first note that $\delta(G,V)$ can only take the values
$0$, $1$ or $p$ if $G$ is a group of order $pm$, where $m$ is
relatively prime to $p$. Then we demonstrate how the precise value is
determined by the fixed point spaces of $V$ and $V^{*}$.

\begin{lemma}\label{deltapm}
 Let $G$ be a group of order $|G|=pm$ such that $p,m$ are coprime. Then for
a $G$-module $V$, we have $\delta(G,V)\in\{0,1,p\}$.
\end{lemma}

\begin{proof}
By \cite[Corollary 2.2]{ElmerKohls2} (which is essentially a reformulation of a
result of Nagata an Miyata \cite{NagataMiyata}), $\delta(G,V)$ is $0$, $1$, or divisible
by $p$. As $\delta(G)$ is the size of a sylow-$p$-subgroup of $G$ by
\cite[Theorem 1,1]{ElmerKohls1}, we also have $\delta(G,V)\le\delta(G)=p$. It
follows $\delta(G,V)\in\{0,1,p\}$.
\end{proof}

For any $G$-module $V$, define
\[
V_{0}:=\{v \in V\mid f(v)=0 \text{ for all } f \in
F[V]^{G}_{1}=(V^{*})^{G}\}=\mathcal{V}((V^{*})^{G}).
\]
Clearly, $V_{0}$ is a $G$-submodule of $V$, because if $v\in V_{0}$, $\sigma\in G$ and
$f\in F[V]^{G}_{1}$ we  have $f(\sigma v)=f(v)=0$, hence $\sigma v\in
V_{0}$.

\begin{lemma}\label{deltaeq1}
For a $G$-module $V$, we have
\[
\delta(G,V)=1 \quad\Leftrightarrow\quad V^{G}\ne\{0\} \text{ and } V^{G}\cap V_{0}=\{0\}.
\]
\end{lemma}

\begin{proof}
Assume that $\delta(G,V)=1$. Then clearly $V^{G}\ne\{0\}$, because
otherwise $\delta(G,V)=0$ by definition.  Take $v\in V^{G}\cap
V_{0}$. If $v\ne 0$, we would have $\epsilon(G,v)=1$, hence there
would be an $f\in F[V]^{G}_{1}$ such that $f(v)\ne 0$, a
contradiction to $v\in V_{0}$. Hence $V^{G}\cap V_{0}=\{0\}$.

Conversely, take a $v\in V^{G}\setminus\{0\}$. By assumption,
$v\not\in V_{0}$, hence there is an $f\in F[V]^{G}_{1}$ such that
$f(v)\ne 0$. Therefore, $\epsilon(G,v)=1$ for all $v\in
V^{G}\setminus\{0\}$ and the claim follows.
\end{proof}

\begin{proposition}
Let $G$ be a group of order $|G|=pm$ such that $p,m$ are coprime.
Then for a $G$-module $V$, we have
\[
\delta(G,V)=\left\{\begin{array}{rl}
0&\text{ if } V^{G}=\{0\}\\
1&\text{ if } V^{G}\ne\{0\} \text{ and } V^{G}\cap V_{0}=\{0\}\\
p&\text{ otherwise. }
\end{array}\right.
\]
\end{proposition}

\begin{proof}
This is immediate from the previous couple of lemmas.
\end{proof}

The benefit of this proposition is that only $V^{G}$ and $(V^{*})^{G}$
need to be known in order to compute $\delta(G,V)$, but not
generators of the full invariant ring $F[V]^{G}$.

\begin{Example}
  Let $G\subseteq S_{p}$ be
any subgroup of order divisible by $p$. Then $G$ contains an element of order $p$, i.e. a $p$-cycle.
Consider the natural action of $G$ on $V:=F^{p}$. Clearly,
$V^{G}=\langle (1,1,\ldots,1)\rangle\subseteq
V_{0}=\mathcal{V}(x_{1}+\ldots+x_{p})$. The proposition implies that
$\delta(G,V)=p$.
\end{Example}

\begin{Example}
 Consider the group
\[
G=\left\{\sigma_{a,b}:=\left(\begin{array}{cc}
1&a\\
0&b
\end{array}\right)\in {\mathbb F}_{p}^{2\times 2} \mid a,b\in {\mathbb F}_{p},\,\, b\ne 0\right\}
\]
of order $|G|=p(p-1)$. Then $G$ acts naturally by left multiplication on the
module  $W:=\langle X,Y\rangle:=F^{2}$ with basis $\{X,Y\}$,
i.e. $\sigma_{a,b}(X)=X$ and $\sigma_{a,b}(Y)=aX+bY$ for all $\sigma_{a,b}\in
G$. Consider the $n$th symmetric power
\[
V_{n}:=S^{n}(W)=\langle e_{0}:=X^{n},\,\, e_{1}:=X^{n-1}Y,\ldots,
e_{n}:=Y^{n}\rangle
\]
with basis $\{e_{0},\ldots,e_{n}\}$. From
\begin{eqnarray*}
\sigma_{a,b}(e_{j})&=&\sigma_{a,b}(X^{n-j}Y^{j})=X^{n-j}(aX+bY)^{j}=\sum_{i=0}^{j}{j
  \choose i}a^{j-i}b^{i}X^{n-i}Y^{i}\\
&=&\sum_{i=0}^{j}{j
  \choose i}a^{j-i}b^{i}e_{i}
\end{eqnarray*}
we see that for $j=1,\ldots,n$, the coefficent of $e_{j-1}$ in $\sigma_{a,b}(e_{j})$ is given by
${j\choose j-1}ab^{j-1}=jab^{j-1}$, which is nonzero if $a=b=1$ and $n<p$. It
follows that $\rank(\sigma_{1,1}-\id_{V_{n}})=n-1$ if $n<p$, and hence
$V_{n}^{\sigma_{1,1}}$ is one dimensional and spanned by $X^{n}$. As
$V_{n}^{G}\subseteq V_{n}^{\sigma_{1,1}}$ and $X^{n}$ is also $G$-invariant, it follows
$V_{n}^{G}=\langle X^{n}\rangle=\langle e_{0}\rangle$. Write
$F[V_{n}]=F[z_{0},\ldots,z_{n}]$, where $z_{i}(e_{j})=\delta_{i,j}$ (the
Kronecker symbol). A similar calculation shows that
$(V_{n}^{*})^{\sigma_{1,1}}=\langle z_{n}\rangle$ if $n<p$, and again we have $(V_{n}^{*})^{G}\subseteq (V_{n}^{*})^{\sigma_{1,1}}$. As
$\sigma_{a,b}(z_{n})=b^{-n}z_{n}$, we see for $1\le n<p$, that is $z_{n}$ is
$G$-invariant only if $n=p-1$. Hence $(V_{n}^{*})^{G}=\{0\}$ for $1\le n\le
p-2$ and $(V_{n}^{*})^{G}=\langle z_{n}\rangle$ if $n=p-1$. In both cases it follows $V_{n}^{G}\subseteq
\mathcal{V}((V_{n}^{*})^{G})$ and hence the proposition above implies $\delta(G,V_{n})=p$ for $1\le n<p$.
\end{Example}

\bibliographystyle{plain}
\bibliography{OurBib}
\end{document}